\documentclass[a4paper,11pt]{article}
\usepackage{amsmath, amssymb,amsxtra}
\usepackage{amsthm} 
\theoremstyle{plain} 
\newtheorem{theorem}{Theorem}[section] 

\newtheorem{corollary}[theorem]{Corollary}
\newtheorem{proposition}[theorem]{Proposition}

\theoremstyle{definition} 
\newtheorem{definition}[theorem]{Definition}
\newtheorem{remark}[theorem]{Remark}

\newtheorem{assumption}[theorem]{Assumption}
\newcommand{\re}{\mathbb{R}}
\newcommand{\com}{\mathbb{C}}
\newcommand{\na}{\mathbb{N}}
\newcommand{\calS}{\mathcal{S}}

\def\dbar{\mbox{\setbox0=\hbox{$d$}$d$\kern-.75\wd0\vbox{%
\hrule height.1ex width.80\wd0\kern1.2ex}}}
\title
{
Singularities for solutions to time dependent Schr\"odinger equations 
with sub-quadratic potential
}
\author
{
Keiichi Kato\footnote{Department of Mathematics, Tokyo University of
Science, 
Kagurazaka 1-3, Shinjuku-ku, Tokyo 162-8601, Japan, 
e-mail: kato@ma.kagu.tus.ac.jp}
 \  and Shingo Ito\footnote{College of Liberal Arts and Sciences, Kitasato University
Kitasato 1-15-1, Minami-ku, Sagamihara, Kanagawa 252-0373, Japan}
}
\date{}
\begin{document}
\maketitle
\abstract{
 In this article, we determine the
wave front sets of solutions to time dependent Schr\"odinger equations 
with a sub-quadratic potential
by using  the representation of the Schr\"odinger
evolution operator via
wave packet transform (short time Fourier transform). 
}
\section{Introduction}
In this article, we consider the following initial value problem 
of the time dependent Schr\"odinger equations,
\begin{equation}
\begin{cases}
 i \partial_t u + \frac{1}{2}\triangle u -V(t,x)u = 0,  &(t,x) \in \re \times 
 \re^n, \\
u(0,x)  = u_0(x), & x \in \re^n,
\end{cases}
\label{LS}
\end{equation}
where $i = \sqrt{-1}$, $u:\re\times \re^n \rightarrow \com$, 
$\triangle = \sum_{j=1}^{n}\frac{\partial^2}{\partial x_j^2}$
and $V(t,x)$ is a real valued function.
\par
We shall determine  the wave front sets of solutions 
to the Schr\"odinger equations \eqref{LS} with a sub-quadratic potential $V(t,x)$
by using 
the representation of the Schr\"odinger evolution operator 
introduced in \cite{K-K-I-1} and \cite{K-K-I-2}
 via the wave packet transform which is defined by A.~C\'ordoba and C.~Fefferman
\cite{C-F-1}. In particular, we determine the location of all 
the singularities of 
the solutions from the information of the initial data. 
\par
We assume the following assumption on $V(t,x)$. 
\begin{assumption}
 \label{ass-V}
$V(t,x)$ is a real valued function in $C^\infty (\re\times \re^n)$ and 
there exists a positive constant $\rho$ such that $0\le \rho <2$ and 
 for all multi-indices $\alpha $,
$$
\left|\partial^\alpha_x V(t,x)\right| \le C (1+|x|)^{\rho -|\alpha |}
$$
holds for some $C>0$ and for all $(t,x)\in \re \times \re^n$. 
\end{assumption}
\par
Let $\varphi \in \calS (\re^n)\backslash \{0\}$ and $f\in
\calS'(\re^n)$. 
We define the wave packet transform $W_{\varphi} f(x,\xi )$ of $f$
with the wave packet generated by a function $\varphi$ as follows: 
$$
W_{\varphi}f(x,\xi ) = \int_{\re^n} \overline{\varphi (y-x)}
f(y)e^{-iy\xi}dy, 
\quad x, \xi \in \re^n. 
$$
In the sequel, we call the function $\varphi$ in the definition of 
wave packet transform {\it basic wave packet}. 
Wave packet transform is called short time
Fourier transform in several literatures(\cite{Grochenig-1}). 
\par
In the previous paper {\cite{K-K-I-1}}, we give a representation of 
the Schr\"odinger evolution 
operator of a free particle, which is the following:
\begin{equation}
  W_{\varphi^{(t)}}u(t, x,\xi ) = 
 e^{-\frac{i}{2}t|\xi |^2}
 W_{\varphi_0}u_0(x-\xi t,\xi ) , 
\label{RPLS}
\end{equation}
where $\varphi^{(t)} = \varphi^{(t)}
(x)=U_0(t)\varphi_0(x)$
with $U_0(t)=e^{i(t/2)\triangle}$, 
$\varphi_0 (x)\in \calS (\re^n)\backslash \{0\}$
and \break $ W_{\varphi^{(t)}}u(t,x,\xi ) =
 W_{\varphi^{(t)} (\cdot)}[u(t,\cdot )](x,\xi )$. In the following, 
we use this convention $W_{\varphi^{(t)}}u(t,x,\xi ) \allowbreak = 
W_{\varphi^{(t)} (\cdot)}[u(t,\cdot )](x,\xi )$ for simplicity, 
if it is not possible to confuse.
 \par
In order to state our results precisely, we prepare several notations. 
Let $b$ be a real number with $0<b<1$. 
For $\varphi_0 (x) \in \calS (\re^n)$, 
we put 
$
\varphi^{(t)}(x) = U_0(t)\varphi_0 (x)
$
with $U_0(t)=e^{i(t/2)\triangle}$, 
$(\varphi_0)_\lambda (x) =\lambda^{nb/2}\varphi_0 (\lambda^{b} x)$
and $\varphi_{\lambda}^{(t)}(x) = 
U_0(t)\left(\varphi_0\right)_\lambda 
(x)$ for $\lambda \ge 1$. 
For $(x_0,\xi_0)\in \re^n\times \re^n\backslash \{0\}$, 
we call a subset $V=K\times \Gamma $ of $\re^{2n}$ 
a conic neighborhood of $(x_0,\xi_0)$ 
if $K$ is a neighborhood of $x_0$ and $\Gamma$ is a conic 
neighborhood of $\xi_0$ (i.e. $\xi \in \Gamma $ and $\alpha >0$
implies $\alpha \xi \in \Gamma$). 
For $\lambda \ge 1 $ and $(x,\xi ) \in \re^n \times \re^n$, 
let $x(s;t,x,\lambda \xi)$ and 
$\xi (s;t,x,\lambda \xi)$ be the solutions 
to 
\begin{equation}
 \begin{cases}
  \dot{x}(s)& = \xi (s), \quad x(t) = x, \\
  \dot{\xi }(s) & = -\nabla V(s,x(s)), \quad \xi (t) = \lambda \xi .
 \end{cases}
\end{equation}
The following theorem is our main result. 
\begin{theorem}
Assume Assumption \ref{ass-V}. 
Take $b=\min \left(\frac{2-\rho}{4}, \frac{1}{4}\right)$. 
 Let $u_0(x) \in L^2 (\re^n)$ and 
$u(t,x)$ be a solution of \eqref{LS} in $C(\re ; L^2(\re^n))$. 
Then under the assumption \ref{ass-V}, 
$(x_0,\xi_0) \notin WF(u(t,x))$
 if and only if 
there exists a conic neighborhood $V = K\times \Gamma $ of
$(x_0,\xi_0)$ such that for all $N\in \na$, for all $a\ge 1$ and 
for all $\varphi_0(x)\in \calS (\re^n )\backslash \{0\}$, 
there exists a constant $C_{N,a, \varphi_0}>0$ satisfying 
\begin{equation}
|W_{\varphi_\lambda^{(-t)}}u_0(x(0;t,x,\lambda\xi ) ,\xi (0;t,x,\lambda\xi ) )| 
\le C_{N,a,\varphi_0} \lambda^{-N}
\label{181007_7Aug14}
\end{equation}
for $\lambda \ge 1$, $a^{-1} \le |\xi |\le a$ and $(x,\xi )\in V$. 
\label{th-1-2}
\end{theorem}
\begin{remark}
 $W_{\varphi_\lambda^{(-t)}}u_0(x,\xi )$ is the wave packet transform 
of $u_0(x)$ with a basic wave packet $\varphi_\lambda^{(-t)}(x)$. 
As previously stated, $\varphi_\lambda^{(-t)}(x)$ depends on $b$. 
\end{remark}
\begin{remark}
 In \cite{K-K-I-2}, the authors investigate the wave front sets 
of solutions to Schr\"odinger equations of a free particle 
and a harmonic oscillator via the wave packet transformation. 
In \cite{K-K-I-4}, the authors give a partial result of the problem 
which is discussed in this paper by the aide of characterization 
of wave front set by G. B. Folland and T. $\bar{\mathrm O}$kaji. 
Characterization of wave front set is discussed in Section 2. 
\end{remark}
\begin{remark}
In one space dimension, 
if $V(t,x)=V(x)$ is super-quadratic in the sense that 
$V(x) \ge C(1+|x|)^{2+\epsilon}$ with some $\epsilon >0$, 
K.~Yajima \cite{Yajima-1} shows that
the fundamental solution of \eqref{LS} has singularities everywhere. 
\end{remark}
\begin{corollary}
Assume Assumption \ref{ass-V} with  $\rho <1$. 
Take $b = \min \left( \frac{1}{4}, 1-\rho \right)$. 
Then $(x_0,\xi_0) \notin WF(u(t,x))$
 if and only if 
there exists a conic neighborhood $V = K\times \Gamma $ of
$(x_0,\xi_0)$ such that for all $N\in \na$, for all $a\ge 1$ and 
for all $\varphi_0(x)\in \calS (\re^n )\backslash \{0\}$,
there exists a constant $C_{N,a,\varphi_0}>0$ satisfying 
$$
|W_{\varphi_\lambda^{(-t)}}u_0(x - \lambda t \xi  ,\lambda \xi )| 
\le C_{N,a,\varphi_0} \lambda^{-N}
$$
for $\lambda \ge 1$, $a^{-1} \le |\xi |\le a$ and $(x,\xi )\in V$. 
\label{cor-1.5}
\end{corollary}
The idea to classify the singularities of generalized functions 
``microlocally'' has been introduced firstly by M.~Sato,  
J.~Bros and  D.~Iagolnitzer and L. H\"ormander 
independently around 1970. 
Wave front set is introduced by L.~H\"ormander in 1970
(see \cite{Hoermander-1}). It is proved in \cite{Hoermander-2} that 
the wave front set of solutions to the linear hyperbolic equations
of principal type propagates along the null bicharacteristics. 
\par
For Schr\"odinger equations, R.~Lascar \cite{Lascar-1} has treated 
singularities of solutions microlocally first. He has introduced 
quasi-homogeneous wave front set and has shown that 
the quasi-homogeneous wave front set of solutions is invariant 
under the Hamilton-flow of Schr\"odinger equation on each 
plane $t=\text{constant}$. 
C.~Parenti and F.~Segala \cite{P-S-1} and T.~Sakurai \cite{Sakurai-1} 
have treated the singularities of solutions to Schr\"odinger equations
in the same way. 
\par
Since the Schr\"odinger operator $i\partial_t +\frac{1}{2}\triangle $
commutes $x+it\nabla $, the solutions become smooth for $t>0$
if the initial data decay at infinity. 
W.~Craig, T.~Kappeler and W.~Strauss \cite{C-K-S-1} have treated 
this type of  smoothing property microlocally. 
They have shown for a solution of \eqref{LS} that 
for a point $x_0 \ne 0$ and a conic neighborhood
$ \Gamma$ of $x_0$, 
$ \langle x \rangle^{r}u_0(x) \in L^2 (\Gamma )$ implies 
$ \langle \xi \rangle^{r}\hat u(t,\xi ) \in L^2 (\Gamma' )$
for a conic neighborhood of $\Gamma' $ of $x_0$ and 
for $t\ne 0$, 
though they have considered more general operators. 
Several mathematicians have shown this kind of results for 
Schr\"odinger operators \cite{Doi-1}, \cite{Doi-2}, \cite{Nakamura-1}, 
\cite{Okaji-1}, \cite{Okaji-2}. 
\par
A.~Hassell and J.~Wunsch \cite{H-W-1} and S.~Nakamura \cite{Nakamura-2}
determine the wave front set of the solution by means of the initial
data. Hassell and Wunsch have studied the singularities by using 
``scattering wave front set''. 
Nakamura has treated the problem in semi-classical way. He has shown
that 
for a solution $u(t,x)$ of \eqref{LS}, 
$(x_0,\xi_0) \notin WF(u(t))$ if and only if 
there exists a $C_0^\infty $ function $a(x,\xi )$ in $\re^{2n}$ 
with $a(x_0,\xi_0)\ne 0$ such that 
$\Vert a(x+tD_x,hD_x)u_0  \Vert = O(h^\infty ) 
\text{ as }h\downarrow 0. $
On the other hand, we use the wave packet transform 
instead of the pseudo-differential operators. 
%
%
\section{Preliminaries}
In this section, we introduce 
the definition of wave front set $WF(u)$ and 
give the characterization of wave front set in terms of wave packet
transform. 
\begin{definition}[Wave front set]
For $f\in \calS' (\re^n)$, we say 
$(x_0,\xi_0) \not\in WF(f)$
 if there exist a function $\chi (x)$ in $C_0^\infty (\re^n)$ with 
$\chi (x_0)\ne 0$ and a conic neighborhood $\Gamma $ of $\xi_0$ such 
that for all $N \in \na $ there exists a positive constant $C_N$ satisfying
$$
|\widehat{\chi f}(\xi )| \le C_N (1+|\xi |)^{-N} 
$$
for all $\xi \in \Gamma $. 
\end{definition}
To prove Theorem \ref{th-1-2}, we use the following characterization 
of the wave front set, which is given in \cite{K-K-I-3}. 
For fixed $b$ with $0<b <1$, 
we put $\varphi_\lambda (x)= \lambda^{nb /2}\varphi (\lambda^{b}x)$. 
\begin{proposition}
Let $(x_0,\xi_0)\in \re^n$ and $u\in \mathcal{S}'(\re^n)$. 
The following conditions are equivalent.
\begin{enumerate}
 \item [(i)] $(x_0,\xi_0)\notin WF(u)$
 \item [(ii)] There exist  $\varphi \in \calS (\re^n )\backslash \{0\}$, 
a conic neighborhood $V$ of $(x_0, \xi_0)$
such that for all $N\in \na$ and for all $a\ge 1$ 
there exists a constant $C_{N,a}>0$ satisfying 
$$
|W_{\varphi_{\lambda}}f(x, \lambda \xi )| 
\le C_{N,a} \lambda^{-N}
$$
for $\lambda \ge 1$ and 
$(x,\xi) \in V$ with $a^{-1} \le |\xi |\le a$. 
 \item[(iii)] 
There exist
a conic neighborhood $V$ of $(x_0, \xi_0)$
such that for all $N\in \na$ and for all $a\ge 1$ 
there exists a constant $C_{N,a}>0$ satisfying 
$$
|W_{\varphi_{\lambda}}f(x, \lambda \xi )| 
\le C_{N,a} \lambda^{-N}
$$
for all $\varphi \in \calS (\re^n )\backslash \{0\}$, 
$\lambda \ge 1$ and $(x,\xi) \in V$ with $a^{-1} \le |\xi |\le a$. 
\end{enumerate}
\label{folland-lemma}
\end{proposition}
\begin{remark}
Characterization of wave front set by wave packet transform is firstly 
given by G. B. Folland \cite{Folland-1}. 
Folland \cite{Folland-1} has shown that the conclusion follows if 
the basic wave packet $\varphi$ is an even and nonzero. 
function in $\calS (\re^n)$ and $b=1/2$. 
P.~G\'erard \cite{Gerard-1} has shown (i) is equivalent to (ii) in Proposition
 \ref{folland-lemma} with  basic wave packet $\varphi (x) = e^{-x^2}$
(Proof is also in J.~M.~Delort \cite{Delort-1}). 
$\bar{\mathrm{O}}$kaji \cite{Okaji-1} has shown that if $\varphi $
 satisfies $\int x^\alpha \varphi (x)dx \ne 0$
for some multi-index $\alpha$. 
\end{remark}
\begin{remark}
Folland \cite{Folland-1} and $\bar{\mathrm{O}}$kaji \cite{Okaji-1} give 
the characterization for $b=1/2$. 
In \cite{K-K-I-3}, we give the characterization for $b=1/2$. 
Without any change of the proof, we can extend for $0<b<1$. 
\end{remark}
%
%
\section{Proofs of Theorem \ref{th-1-2} and  Corollary \ref{cor-1.5}}
In this section, we prove Theorem \ref{th-1-2} and Corollary \ref{cor-1.5}. 
\par
\begin{proof}[Proof of Theorem \ref{th-1-2}]
The initial value problem \eqref{LS} is transformed 
by the wave packet transform to 
\begin{equation}
\begin{cases}
 & \left(i\partial_t + i\xi\cdot \nabla_x -i\nabla_x V(t,x) \cdot\nabla_{\xi}
 -\frac{1}{2}|\xi |^2-\widetilde{V}(t,x)\right) \times\\
&\phantom{xxxxxxxxxxxxxxxxxxxxxxxxxxx} 
W_{\varphi^ {(t)}}u(t,x,\xi )
 =Ru(t,x,\xi ), \\
&W_{\varphi^{(0)}}u(0,x,\xi ) =  W_{\varphi_0}u_0(x,\xi ),
\end{cases}
\label{WPEQP}
 \end{equation}
where
$\widetilde{V}(t,x)= V(t,x)- \nabla_x V(t,x)\cdot x$ and 
\begin{multline*}
 Ru(t,x,\xi ) = \sum_{|\alpha |=2}\frac{1}{\alpha !}\int 
 \overline{\varphi^{(t)} (y-x)}\\
\times 
 \left(
 \int_{0}^{1}\partial^\alpha V(t,x+\theta (y-x))
 (1-\theta)d\theta
 \right)
 (y-x)^\alpha u(t,y)e^{-i\xi y}dy. 
\end{multline*}
Solving \eqref{WPEQP}, we have the integral equation 
\begin{multline*}
\label{Int-eq-Sch}
 W_{\varphi^{(t)}}u(t,x,\xi ) = 
e^{-i\int_{0}^{t}\{\frac{1}{2}|\xi (s;t,x,\xi )|^2+
\widetilde{V}(s,x(s;t,x,\xi))\}ds}
W_{\varphi_0}u_0(x(0;t,x,\xi),\xi (0;t,x,\xi))\\
-i\int_{0}^{t}
e^{-i\int_{s}^{t}\{\frac{1}{2}|\xi
 (s_1;t,x,\xi)|^2+\widetilde{V}(s_1,x(s_1;t,x,\xi))\}ds_1}
Ru(s,x(s;t,x,\xi),\xi (s;t,x,\xi))ds,
\end{multline*}
where $x(s;t,x,\xi)$ and $\xi (s;t,x,\xi)$ are the solutions of 
\begin{equation*}
 \begin{cases}
  \dot{x}(s)&= \xi (s), \ x(t)=x, \\
 \dot{\xi}(s)&= -\nabla_x V(s,x(s)), \ \xi (t)= \xi. 
 \end{cases}
\end{equation*}
\par
For fixed $t_0$, we have 
\begin{multline}
\label{int-eq-1}
 W_{\varphi_{\lambda }^{(t-t_0)}}
u(t,x(t;t_0,x,\lambda \xi ),\xi (t;t_0,x ,\lambda \xi ) ) \\ 
= 
e^{-i\int_{0}^{t}\{\frac{1}{2}|\xi (s;t_0,x,\lambda \xi)|^2
+\widetilde{V}(s,x(s;t_0,x, \lambda \xi ))\}ds}
W_{\varphi_{\lambda }^{(-t_0)}}u_0(x(0;t_0,x, \lambda \xi),
\xi (0;t_0,x, \lambda \xi ))\\
-i\int_{0}^{t}
e^{-i\int_{s}^{t}\{\frac{1}{2}
|\xi(s_1, t_0,x,\lambda \xi )|^2
+\widetilde{V}(s_1,x(s_1;t_0,x, \lambda \xi ))\}ds_1}
Ru(s,x(s;t_0,x, \lambda \xi ),\xi (s;t_0,x,\lambda \xi ))ds, 
\end{multline}
substituting 
$(x(t;t_0,x, \lambda \xi ), \xi (t;t_0,x, \lambda \xi ))$ 
and $\varphi_\lambda ^{(-t_0)}(x)$
for $(x,\xi )$ and $\varphi_0 (x)$ respectively. 
Here we use the fact that 
\begin{align*}
 x(s;t,x(t;t_0,x,\lambda \xi ), \xi(t;t_0,x,\lambda \xi))
 &=x(s;t_0, x,\lambda \xi ),\\
 \xi (s;t,x(t;t_0,x,\lambda \xi ), \xi(t;t_0,x,\lambda \xi))
 &=\xi (s;t_0, x,\lambda \xi )
\end{align*}
and $e^{\frac{i}{2}t\triangle} \varphi_{\lambda}^{(-t_0)}(x)
= \varphi_{\lambda }^{(t-t_0)}(x)$. 
\par
We fix $a\ge 1$. 
Let $V= K\times \Gamma $ be a neighborhood of $(x_0,\xi_0)$ satisfying 
\eqref{181007_7Aug14} for $\lambda \ge 1 $, $a^{-1}\le |\xi |\le a$ and 
$(x,\xi )\in V$. 
We only show the sufficiency here because the necessity is proved 
in the same way. To do so, we show that 
the following assertion 
$P(\sigma , \varphi_0)$ holds for all $\sigma \ge 0$ 
and for all $\varphi_0 \in \calS (\re^n)\backslash \{0\}$. 
\newline
$P(\sigma , \varphi_0)$: 
``
For $a\ge 1$ there exists
a positive constant $C_{\sigma, a, \varphi_0}$ 
such that 
\begin{equation}
|W_{\varphi^{(t-t_0)}_\lambda }
u(t, x(t;t_0,x,\lambda \xi),\xi (t;t_0,x,\lambda \xi) )| \le
C_{\sigma ,a, \varphi_0}\lambda ^{-\sigma}
\label{ineq-6}
\end{equation}
for all $x \in K $, all $\xi \in \Gamma$
with $1/a \le |\xi | \le a $, all $\lambda \ge 1$
and $0 \le t \le t_0$. ''
\par
In fact, taking $t=t_0$, we have $\varphi^{(t_0-t_0)}_\lambda
 =\left(\varphi_0\right)_\lambda $, 
$x(t_0;t_0,x,\lambda \xi)= x$ and 
$\xi (t_0;t_0,x,\lambda \xi)=\lambda\xi$. Hence 
from \eqref{ineq-6}, we have immediately
$$
|W_{(\varphi_0)_\lambda}
u(t_0, x,\lambda\xi )| \le
C_{\sigma ,a, \varphi_0}\lambda ^{-\sigma}
$$
for $\lambda \ge 1$, $x\in K$ and $\xi \in \Gamma$ with $1/a\le |\xi
 |\le a$. 
 This and Proposition \ref{folland-lemma} show the sufficiency.
\par
We write $x^* = x(s;t_0,x,\lambda\xi ), \xi^* = \xi (s;t_0,x,\lambda \xi
)$, $t^* = s-t_0$ and 
$\varphi_\lambda (x)=\left(\varphi_0\right)_\lambda (x) $
 for simple description. 
\par
We show by induction with respect to $\sigma $
that $P(\sigma , \varphi_0)$ holds for all 
$\sigma \ge 0$ and for all $\varphi_0 \in \calS (\re^n )\backslash
 \{0\}$.
\par
First we show that $P(0, \varphi_0)$ 
holds for all $\varphi_0 \in \calS (\re^n )$. 
Since $u_0(x)\in L^2(\re^n)$, $u(t,x) \in C(\re ; L^2(\re^n))$. 
Schwarz's inequality and conservativity for  $L^2$ norm of solutions of
\eqref{LS} show that 
\begin{align*}
& \left|W_{\varphi_\lambda^{(t-t_0)} }
u(t, x(t;t_0,x,\lambda \xi ),\lambda \xi(t;t_0,x,\lambda \xi ) )\right| \\
&
\le \int |\varphi_\lambda^{(t-t_0)} (y-x(t;t_0,x,\lambda \xi ))||u(t,y)|dy \\
& \le \Vert \varphi_\lambda ^{(t-t_0)}(\cdot )\Vert_{L^2}
\Vert u(t, \cdot )\Vert_{L^2} \\
& = \Vert \varphi_\lambda (\cdot ) \Vert_{L^2}
\Vert u_0(\cdot )\Vert_{L^2}
= \Vert \varphi_0(\cdot ) \Vert_{L^2}
\Vert u_0(\cdot )\Vert_{L^2}.
\end{align*}
Hence $P(0, \varphi_0)$ holds. 
\par
Next we show that for fixed $\varphi_0 \in \calS (\re^n)\backslash \{0\}$, 
$P(\sigma +2b , \varphi_0)$ holds under 
the assumption that $P(\sigma, \varphi_0)$ holds
for all $\varphi_0\in \calS (\re^n )\backslash \{0\}$. 
To do so, it suffices to show that for fixed $\varphi_0$, 
there exists a positive constant $C_{a, \varphi_0}$ such that 
\begin{equation}
 |Ru(s,x(s;t_0,x,\lambda \xi), \xi (s;t_0,x,\lambda \xi))|
\le C_{a,\varphi_0}\lambda^{-(\sigma + 2b)}
\label{ineq-7}
\end{equation}
for all $x \in K $, all $\xi \in \Gamma$
with $1/a \le |\xi | \le a $, all $\lambda \ge 1$
and $0 \le s \le t_0$, since the first term of the right hand
side of \eqref{int-eq-1} is estimated by the condition on $u_0$. 
\par
Let $L$ be an integer. Taylor's expansion of $V (s, y)$ yields that 
\begin{multline}
 Ru(s, x^*,\xi^*) 
\\
= \sum_{2\le |\alpha |\le L-1} \frac{\partial_x^\alpha V(s,x^*)}{\alpha
 !}
 \int (y-x^*)^\alpha 
 \overline{\varphi_\lambda ^{(s-t_0)}(y-x^*)}
u(s,y) e^{-iy\xi^* }
dy + R_L,
\label{eq-10}
\end{multline}
where 
\begin{multline*}
 R_L(s,x^*, \xi^*) = L\sum_{|\alpha |=L}\frac{1}{\alpha !}
\frac{1}{\Vert \varphi_0\Vert^2_{L^2}}
\\
\times \iint 
\left(
\int 
\left(
\int_{0}^{1}\partial_x^\alpha V(s,x^*-\theta (x^* -y))(1-\theta
 )^{L-1} d\theta 
\right)
\right.
 (y-x^*)^\alpha 
\\
\times
 \overline{\varphi_\lambda^{(s-t_0)}(y-x^*)}
\left.
\varphi_\lambda^{(s-t_0)} (y-z) 
e^{-iy (\xi^* -\eta )}
dy
\right)
 W_{\varphi_\lambda^{(s-t_0)}}u(s, z, \eta ) 
 dz d\eta .
\end{multline*}
Here we use the inversion formula of the wave packet transform
$$
\frac{1}{\Vert \varphi \Vert^2_{L^2}}
W_{\varphi}^{-1} W_{\varphi}f (x) = f(x),
$$
where 
$$
W_{\varphi}^{-1} g (x) = \iint g(y,\xi )\varphi (y-x) e^{ix\xi} d\xi dy
$$
for a smooth tempered function $g(y,\xi )$ on $\re^{2n}$. 
\par
The strategy for the proof of \eqref{ineq-7} is the 
following. In Step 1, taking 
$b=\frac{1}{4}\min (2-\rho ,1 )$
 according to the value of $\rho$ which is the order of increasing of
$V(t,x)$ with respect to $x$ 
in the assumption \ref{ass-V}, we estimate the first term 
of the right hand side of \eqref{eq-10}. In Step 2, taking $L$
 sufficiently large according to the value of $\sigma$, 
we estimate the second term $R_L$ of the right 
hand side of \eqref{eq-10}. 
\par
(Step1) We estimate the first term of the right hand side of
 \eqref{eq-10}. 
Let $U_0(t)= e^{\frac{i}{2}t\triangle}$. 
Since 
$ x U_0(t) = U_0(t)(x-it\nabla )$, 
we have 
\begin{align*}
(y - x^*)^\alpha 
\varphi_\lambda ^{(t^*)}(y-x^*)
&= U_0(t^*)\left[
(y- x^*- it^* \nabla_y )^\alpha (\varphi_0)_\lambda
\right]
\\
= & \sum_{\beta + \gamma \le \alpha}
C_{\beta,\gamma}{t^*}^{|\beta |}\lambda^{b(|\beta |-|\gamma |) }
\varphi^{(\beta,\gamma)}_\lambda (t^*,y-x^*),
\end{align*}
where $\varphi^{(\beta,\gamma)}(x)= x^\gamma \partial_x^\beta
\varphi_0(x)$ 
and $\varphi^{(\beta,\gamma)}_\lambda (t,x)=
U_0(t)\left(\varphi^{(\beta, \gamma)}\right)_\lambda
 (x)$. 
The assumption of induction yields that 
\begin{align*}
& |(\text{The first term of the right hand side of \eqref{eq-10}}) |\\
& \le \sum_{2\le |\alpha |\le L-1}\sum_{\beta + \gamma =\alpha}
\frac{1}{\alpha !}| \partial_x^\alpha V(s,x^*) |
C_{\beta,\gamma}{t^*}^{|\beta |}\lambda^{b(|\beta |-|\gamma |) }
\left|W_{\varphi^{(\beta,\gamma)}_\lambda (t^*,x)}u(s,x^*,\xi^*)\right|
\\
&\le \sum_{2\le |\alpha |\le L-1}\sum_{\beta + \gamma =\alpha}
\frac{1}{\alpha !}C(1+|x^*|)^{\rho -|\alpha |}
C_{\beta,\gamma}{t^*}^{|\beta |}\lambda^{b(|\beta |-|\gamma |) }
C \lambda^{-\sigma } .
\end{align*}
Since 
\begin{multline}
 x^* = x(s;t_0,x,\lambda \xi ) = 
 x+ \int_{t_0}^{s}\dot{x}(s_1)ds_1 \\
 = x+ (s-t_0)\lambda \xi - \int_{t_0}^{s}(s-s_1)\nabla_x V(s_1,
 x(s_1))ds_1,
\label{eq-9}
\end{multline}
there exists a positive constant $\lambda_0$ such that 
\begin{equation}
\label{ode-est}
 |x^*| \ge \frac{1}{2a}|t^*|\lambda
\end{equation}
for all $\lambda \ge \lambda_0$, $\lambda^{-2b}\le |t^*| \le t_0$, 
$x\in K$ and $\xi \in \Gamma $ with $1/a \le |\xi |\le a$. 
( see Appendix A
 for the proof of \eqref{ode-est}).
 Hence we have for $\lambda^{-2b}\le |t^*| \le t_0$
\begin{align*}
& |(\text{The first term of the right hand side of \eqref{eq-10}}) |\\
&\le \sum_{2\le |\alpha |\le L-1}\sum_{\beta + \gamma =\alpha}
\frac{1}{\alpha !}C(1+|t^*|\lambda )^{\rho -|\alpha |}
C_{\beta,\gamma}{t^*}^{|\beta |}\lambda^{b(|\beta |-|\gamma |) }
C \lambda^{-\sigma } \\
&\le C'\sum_{2\le |\alpha |\le L-1}
\left(\lambda^{\rho -|\alpha |+b|\alpha |} + \min (\lambda
 ^{-1},\lambda^{2-\rho })+ \lambda^{-2b}\right)\lambda^{-\sigma}
\\ 
&\le C'' \lambda^{-2b -\sigma }, 
\end{align*}
since $2 b= \frac{1}{2}\min (2 -\rho, 1 )$. 
For $|t^*|< \lambda^{-2b}$, we have that
\begin{align*}
& |(\text{The first term of the right hand side of \eqref{eq-10}}) |\\
&\le \sum_{2\le |\alpha |\le L-1}\sum_{\beta + \gamma =\alpha}
\frac{1}{\alpha !}C
C_{\beta,\gamma}{t^*}^{|\beta |}\lambda^{b(|\beta |-|\gamma |) }
C \lambda^{-\sigma } \\
&\le \sum_{2\le |\alpha |\le L-1}\sum_{\beta + \gamma =\alpha}
\frac{1}{\alpha !}C
C_{\beta,\gamma}\lambda^{-b(|\beta |+|\gamma |) }
C \lambda^{-\sigma } 
= C' \lambda^{-2b -\sigma }.
\end{align*}
\par
(Step 2) We estimate $R_{L}$. 
Let $\psi_1, \psi_2$ be $C^\infty $ function on $\re$ satisfying 
\begin{align*}
&  \psi_1(s) = 
\begin{cases}
 1& \quad \text{for } s\le 1, \\
 0& \quad \text{for } s \ge 2, 
\end{cases}
\\
& \psi_2(s)  =
\begin{cases}
 0& \quad \text{for } s\le 1, \\
 1 & \quad \text{for } s \ge 2, 
\end{cases}
\\
&\psi_1(s) + \psi _2(s) =1 \quad \text{for all } s\in \re. 
\end{align*}
Take $d $ with $0 < d< b$. 
Putting $ V_\alpha (s,x^*,y)=
\int_{0}^{1}\partial_x^\alpha V(s,x^*-\theta (x^* -y))(1-\theta
 )^{L-1} d\theta $ and 
\begin{multline*}
 I_{\alpha,j} (s,x^*, \xi^*, \lambda )
 =\iiint
\psi_j \left(
\frac{\lambda^d |y-x^*|}{1+\lambda |t^*|} \right)
 V_\alpha (s,x^*,y) (y- x^*)^\alpha
\\
 \overline{\varphi_\lambda ^{(t^*)}(y-x^*)}
 \varphi_\lambda^{(t^*)}(y-z)
 W_{\varphi_\lambda^{(t^*)}}u(s, z, \eta ) e^{-iy (\xi^* -\eta )}
 dz d\eta dy
\end{multline*}
for $j=1,2$, 
we have 
\begin{equation}
 R_{L}(s,x^*,\xi ^*,\lambda ) =
L \sum_{|\alpha |=L}\frac{1}{\alpha !}
\frac{1}{\Vert \varphi_0\Vert_{L^2}^2}
\sum_{j=1}^{2}
 I_{\alpha,j} (s,x^*, \xi^*, \lambda ). 
\end{equation}
\par
We need to show that for $j=1,2$, 
there exists a positive constant $C_{\sigma, a,\varphi_0}$ 
such that 
\begin{equation}
 | I_{\alpha, j} (s,x^*, \xi^*, \lambda ) | \le
C_{\sigma, a,\varphi_0} \lambda ^{-\sigma -2b}
\label{ineq-9}
\end{equation}
for 
$\lambda \ge 1$, 
$x\in K$, $\xi \in \Gamma $ with $1/a \le |\xi |\le a$ 
and $0\le s\le t_0$. 
For $I_{\alpha,1,1}$, 
integration by parts and 
the fact that 
$(1-\triangle_y)e^{iy(\xi -\eta )} = (1+|\xi -\eta|^2)e^{iy(\xi -\eta)}$
yield that 
\begin{multline*}
 I_{\alpha, 1} (s,x^*, \xi^*, \lambda ) 
= \iiint 
\left(
1+ |\xi -\eta|^2
\right)^{-N}\\
\times (1-\triangle_y)^N
\left[
 \overline{\varphi_\lambda ^{(t^*)}(y-x^*)}
 \varphi_\lambda^{(t^*)}(y-z)
\psi_j \left(
\frac{\lambda^d |y-x^*|}{1+\lambda |t^*|} \right)
\right.
\\
\left.
\times 
 V_\alpha (s,x^*,y)(y-x^*)^\alpha \right] 
 W_{\varphi_\lambda^{(t^*)}}u(s, z, \eta ) e^{-iy (\xi^* -\eta )}
 dy d\eta dz. 
\end{multline*}
We take $d'$ with $0<d'<d$. 
Since
$|y-x^* | \le 2 
(1+\lambda |t^*|) \lambda^{-d }$ in 
the support of 
$
\psi_1 \left(
\frac{\lambda^d |y-x^*|}{1+\lambda |t^*|} \right)
$
with respect to $y$, 
the estimate \eqref{ode-est} shows that for $|t^*|\ge \lambda^{d'-1}$
and $\lambda \ge \lambda_0$ with some $\lambda_0 \ge 1$. 
\begin{align*}
 |\partial_x^\alpha V(s, x^* + \theta (y-x^*))||(y- x^*)^\alpha |&
\le C (1+ |x^* +\theta (y-x^*)|)^{\rho -L} 
(1+\lambda |t^*|)^{L} \lambda^{-d L}\\
&
\le C (1+ |x^*|-|y-x^*|)^{\rho -L} 
(1+\lambda |t^*|)^{L} \lambda^{-d L}\\
&
\le C (1+\lambda |t^*|)^{\rho } \lambda^{-d L}, 
\end{align*}
from which we have 
\begin{equation}
 |I_{\alpha,1}(s,x^*, \xi^*, \lambda ) |   \le 
C \lambda^{-d L} \lambda ^{l}, 
\label{ineq-13}
\end{equation}
where $l$ are positive numbers which are independent
of $L$. 
For $|t^*| \le \lambda^{d'-1}$, we have 
$|y-x^*|\le C(1+\lambda |t^*|)\lambda^{-d}
\le C \lambda^{1-d'+d-1}\le C\lambda^{d-d'}$, 
which shows \eqref{ineq-13} for some $l$. 
Hence \eqref{ineq-9} with $j=1$ holds if we take 
$L$ sufficiently large. 
\par
Finally we estimate $I_{\alpha,2}$.
Since $x U_0(t)= U_0(t)(x-it \nabla_x)$, $\partial_{x_j}U_0(t) =
 U_0(t)\partial_{x_j}$, 
$x\varphi_\lambda (x) = \lambda^{-b}(x\varphi)_\lambda (x)$ and 
$\nabla \varphi_\lambda (x) = \lambda^{b}(\nabla \varphi)_\lambda (x)$, 
we have for an integer $M$ and a multi-index $\alpha $
\begin{align}
 (1+ |x|^2)^M \partial_x^\alpha \varphi_\lambda^{(t)}(x)
& = U_0(t)\left[
(1+|x-it \nabla |^2)^M \partial_x^\alpha 
\varphi_{0,\lambda }(x)
\right]\label{204501_23Jun14}\\
=& U_0(t)\left[
\sum_{|\beta + \gamma |\le 2M}
C_{\beta , \gamma }(\lambda^{b}t)^{|\gamma |}
 \lambda^{-b(|\beta | - |\alpha |)}
(x^\beta \partial^\gamma_x\varphi_0)_\lambda 
\right]\label{204511_23Jun14}  \\
\le & 
\sum_{|\beta + \gamma |\le 2M}
C_{\beta , \gamma }(\lambda^{b}t)^{|\gamma |}
 \lambda^{-b(|\beta | - |\alpha |)}
 U_0(t)\left[
(x^\beta \partial^\gamma_x\varphi_0)_\lambda 
\right] .\label{204536_23Jun14}
\end{align}
Hence we have for $M,N \in \na$, 
\begin{align*}
&|I_{\alpha ,2}| \\
 = &\left|
\iiint 
\psi_2 \left(
\frac{\lambda^d |y-x^*|}{1+\lambda |t^*|} \right)
 V_\alpha (s,x^*,y) (y -x^* )^\alpha 
 \overline{\varphi_\lambda ^{(t^*)}(y-x^*)}
 \varphi_\lambda ^{(t^*)} (y-z)
 W_{\varphi_\lambda^{(t^*)}}u(s,z,\eta )
 e^{-iy (\xi^* -\eta )} dy
\right|
\\
 = &\left|
\iiint 
(1+|y-x^*|^2)^{-M} (1+ |\eta - \xi^* |^2)^{-N}
(1+|y-x^*|^2)^{M} 
\right. \\
\times &
(1-\triangle_y)^{N}\left[
\psi_2 \left(
\frac{\lambda^d |y-x^*|}{1+\lambda |t^*|} \right)
 V_\alpha (s,x^*,y) (y-x^*)^\alpha 
 \overline{\varphi_\lambda ^{(t^*)}(y-x^*)}
 \varphi_\lambda ^{(t^*)} (y-z)
 W_{\varphi_\lambda^{(t^*)}}u(s,z,\eta )
\right]
\\
& \phantom{xxxxxxxxxxxxxxxxxxxxxxxxxxxxxxxxxxxxxxxx}
\left.
\times e^{-iy (\xi^* -\eta )} dy
\right|
\\
\le & \sum_{|\alpha_1+ \cdots +\alpha_4|\le 2N}
\sum_{|\beta + \gamma |\le 2M+ |\alpha |
}\sum_{\alpha_3' \le \alpha_3}
C_{\alpha_1, \ldots , \alpha_4, \beta, \gamma, \alpha_3'}
(\lambda ^b |t^*|)^{|\gamma |}\lambda^{b(|\alpha_1 | - |\beta |)}
\\
&
\phantom{xx}
\times \iiint 
(1+|y-x^*|^2)^{-M} (1+ |\eta - \xi^* |^2)^{-N}
\left|
U_0(t^*)\left[
\left(x^\beta \partial_y^{\alpha_1+\gamma }\varphi_0 \right)_{\lambda }
\right](y-x^*)
\right|
\\
&
\times 
\left|
U_0(t^*)\left[
\left(\partial_y^{\alpha_3}\varphi_0 \right)_{\lambda }
\right](y-z)
\right|
(1+\lambda |t^*|)^{-|\alpha_3|}\lambda^{d|\alpha_3|}
\left|
\partial_x^{\alpha_3'}\psi_2
\right|
\left|\partial_y^{\alpha_4}V_\alpha \right|
\left| W_{\varphi_\lambda^{(t^*)}}u(s,z,\eta )\right|
dz d\eta dy
\end{align*}
\par
Since $|y-x^*| \ge \lambda^{-d}(1+\lambda |t^*|)$ in the support
 of $\psi_1(\lambda^{-d}|y-x^*|/(1+|t^*|\lambda ))$, 
we have with $M=m+n+1$ and $N=n+1$
\begin{multline*}
 |I_{\alpha ,2}| \le \sum_{|\alpha_1+ \cdots +\alpha_4|\le 2N}
\sum_{|\beta + \gamma |\le 2M+ |\alpha |
}\sum_{\alpha_3' \le \alpha_3}
C (1+ \lambda^{-2d}(1+\lambda |t^*|)^2)^{-m}
\Vert (1+|\cdot |^2)^{-n-1})\Vert_{L^2_y}
\\
\times \Vert (1+|\cdot |^2)^{-n-1})\Vert_{L^2_\eta}
\Vert x^\beta \partial_y^{\alpha_1+\gamma}\varphi_0\Vert_{L^2_y}
\Vert \partial_z \varphi_0\Vert_{L^2_z}
\times \Vert W_{\varphi_\lambda^{(t^*)}}u(s,z,\eta
 )\Vert_{L^2_{z,\eta}}. 
\end{multline*}
For $0\le t\le \lambda^{-2b}$, 
we have 
$$
|I_{\alpha ,2}| \le C \lambda^{-Mb}\lambda^{b(n+1+L)}\lambda^{-\sigma}
= C\lambda^{-b(M-n-1-L)+\sigma}
\le C \lambda^{-2b-\sigma}, 
$$
if we take $M \ge n+3+L$. 
For  $\lambda^{-2b}\le t \le t_0$, 
\begin{align*}
 |I_{\alpha ,2}| &\le C 
(1+ (\lambda^{1-d-2b})^2)^{-m}
\lambda^{b(2M+2N+L)}\lambda^{-\sigma}
\\
& \le C 
\lambda^{-2m(1-d-2b)}\lambda^{b(2m+4(n+1)+L)}\lambda^{-\sigma}
\\
& \le C 
\lambda^{-2m(1-d-3b)}\lambda^{b(4(n+1)+L)}\lambda^{-\sigma}. 
\end{align*}
Since $1-d-2b > 1-4b\ge 0$, we have $|I_{\alpha,2}|\le C
 \lambda^{-2b-\sigma}$, 
if we take $m$ sufficiently large. 
This shows \eqref{ineq-9} with $j=2$ for 
$x \in K $, $\xi \in \Gamma$ with $1/a \le |\xi | \le a $ and 
$\lambda \ge 1$ and $0\le s\le t_0$. 
\end{proof}
\begin{proof}[Proof of Corollary \ref{cor-1.5}]
\eqref{eq-9} shows that 
\begin{equation}
 x(0;t,x,\lambda \xi) = x-\lambda t\xi + \delta_1(\lambda )
\label{eq-12}
\end{equation}
with $\delta_1(\lambda)=O(\lambda^{\rho -1})$. 
In the same way as for \eqref{eq-12}, we have 
\begin{equation}
 \xi (0;t,x,\lambda \xi) = \lambda \xi + \delta_2(\lambda )
\label{eq-13}
\end{equation}
with $\delta_2(\lambda )=O(\lambda^{\rho -1})$. 
We show that 
\begin{equation}
W_{\varphi_{\lambda}^{(t^*)}}u_0(x-\lambda t\xi + \delta_1(\lambda ), 
\lambda \xi + \delta_2(\lambda ) )=
W_{\varphi_{\lambda}^{(t^*)}}u_0(x -\lambda t\xi , \lambda \xi ) 
+ (\text{lower order term}). 
\label{154101_30Jul14}
\end{equation}

We have
\begin{align*}
& W_{\varphi_{\lambda}^{(t^*)}}u_0(x-\lambda t\xi + \delta_1(\lambda ), 
\lambda \xi + \delta_2(\lambda ) ) \\
& = \int \varphi_{\lambda}^{(t^*)}
(y - (x-\lambda \xi t +\delta_1(\lambda ))) 
u_0(y)e^{-iy(\lambda \xi + \delta_2(\lambda ))}dy .
\end{align*}
By Taylor's expansion, we have with an integer $L$
 \begin{multline*}
   \varphi_{\lambda}^{(t^*)}(y - (x-\lambda \xi t +\delta_1(\lambda ))) 
= \varphi_{\lambda}^{(t^*)}(y - (x-\lambda \xi t))  \\
+ \sum_{1 \le |\alpha | \le L} \frac{1}{\alpha !}
\partial_x^{\alpha}\left( 
\varphi_{\lambda}^{(t^*)}(y - (x-\lambda \xi t )) 
\right)
\left(-\delta_1(\lambda )\right)^\alpha \\
+\sum_{|\alpha | =L+1} \frac{1}{\alpha !}
r_{\alpha }
\left(-\delta_1(\lambda )\right)^\alpha 
 \end{multline*}
and 
$$
e^{-y(\lambda \xi + \delta_2(\lambda ))}
= 
e^{-y\lambda \xi} 
\left(
1+ \sum_{1 \le |\alpha |}
\frac{1}{\alpha !}
\left(-i y \delta_1(\lambda )\right)^\alpha 
\right), 
$$
from which we obtain 
\begin{multline*}
 W_{\varphi_{\lambda}^{(t^*)}}u_0(x-\lambda t\xi + \delta_1(\lambda ), 
\lambda \xi + \delta_2(\lambda ) ) = \\
W_{\varphi_{\lambda}^{(t^*)}}u_0(x-\lambda t\xi , \lambda \xi ) \\
+ 
\sum_{1\le |\alpha | \le L }\sum_{1\le |\beta | }
\lambda^{b|\alpha |}\frac{(-\delta_1)^\alpha}{\alpha !}
\frac{(-\delta_2)^\beta }{\beta !}
W_{(\partial _x^\alpha \varphi)_\lambda^{(t^*)}}
\left[
y^\beta u(y)
\right]
(x-\lambda t\xi , \lambda \xi )
\\
+ 
\sum_{|\alpha | = L+1 }\sum_{1\le |\beta |}
\lambda^{b|\alpha |}\frac{(-\delta_1)^\alpha}{\alpha !}
\frac{(-\delta_2)^\beta }{\beta !}
\int R_\alpha 
y^\beta u(y)
e^{-iy\lambda \xi }dy. 
\end{multline*}
This implies \eqref{154101_30Jul14} with large $L$, since 
$|\delta_1 (\lambda )|, |\delta_2 (\lambda )| \le \lambda^{\rho -1}$,   
$W_{(\partial _x^\alpha \varphi)_\lambda^{(t^*)}}
\left[
y^\beta u(y)
\right]
(x-\lambda t\xi , \lambda \xi )$ is the same order of 
$
W_{\varphi_{\lambda}^{(t^*)}}u_0(x-\lambda t\xi , \lambda \xi )
$
with respect to $\lambda$ and the order of 
$\int R_\alpha 
y^\beta u(y)
e^{-iy\lambda \xi }dy$ with respect to $\lambda $
is estimated above by some constant. 
\end{proof}
%
%
\appendix
\section{Proof of the estimate \eqref{ode-est}}
In this appendix, we give the proof of the estimate \eqref{ode-est}. 
We fix $p$. We show the estimate \eqref{est-15} 
for $|t_0|\ge |t^*|\ge \lambda^{p-1}$, $\lambda \ge \lambda_0$, 
$x\in K$, $\xi \in \Gamma $ with $1/a \le |\xi |\le a$.
\begin{proof}
 The equation \eqref{eq-9} can be solved by Picard's iteration method. 
 We put $x^{(0)}(s)= x+(s-t_0)\lambda \xi $ and we define 
 $$
 x^{(N+1)}(s) = x+(s-t_0)\lambda \xi 
 - \int_{t_0}^{s}(s-s_1)\nabla_x V(s_1, x^{(N)}(s_1))ds_1
 $$
 for $N \ge 0$. Then we have the solution $x(s)$ of \eqref{eq-9} as 
 $x(s)=\lim_{N\rightarrow \infty}x^{(N)}(s)$.  
 We show that there exists a positive constant $\lambda_0\ge 1$
 such that 
 \begin{equation}
 \frac{1}{2a} |t^*|\lambda \le  |x^{(N)}(s)|
 \le 2 a|t^*|\lambda , 
 \label{est-15}
 \end{equation}
 for $\lambda \ge \lambda_0$, $\lambda^{p-1}\le |t^*| \le t_0$, 
$x\in K$ and $\xi \in \Gamma $ with $1/a \le |\xi |\le a$. 
 We only treat the case that $1\le \rho <2$. 
 We show \eqref{est-15} by induction with respect to $N$. 
 \par
Obviously \eqref{est-15} holds for $N=0$. 
 \par
Assuming that \eqref{est-15} holds for $N$, we have
 \begin{align*}
 |x^{(N+1)}(s)| & \ge |x+(s-t_0)\lambda \xi| -
 \left|
 \int_{t_0}^{s}|s-s_1||\nabla_x V(s_1, x^{(N)}(s_1))|ds_1
 \right| \\
 & \ge |t^*|\lambda |\xi| -|x|-
 \int_{s}^{t_0}|s-s_1|C(1+|x^{(N)}(s_1)|)^{\rho -1}ds_1
 \\
 & \ge |t^*|\lambda |\xi| -|x|-
 C \int_{s}^{t_0}|s-s_1|
(1+2(|t_0-s_1|\lambda |\xi|)^{\rho -1})ds_1
 \\
 & \ge |t^*|\lambda |\xi| -|x|-
 C |t^*|^2
-C \lambda^{\rho -1}|\xi |^{\rho -1}|t^*|^{\rho +1}
 \\
 & \ge |t^*|\lambda |\xi|
 \left(
 1- \frac{|x|}{|t^*|\lambda |\xi|}
- C \frac{|t_0|}{\lambda |\xi |}
-C|t_0|^\rho \lambda^{\rho -2} |\xi |^{\rho -2}
 \right) \\
 & \ge |t^*|\lambda |\xi|
 \left(
 1- \frac{a|x|}{\lambda^{p}}
 -C \frac{a|t_0|}{\lambda}
-C\frac{a^{2-\rho}|t_0|^\rho }{\lambda^{2-\rho}}
 \right). 
 \end{align*}
Since $p>0$ and $2-\rho >0$, 
there exists a constant $\lambda_0 \ge 1$ such that 
 $$
 1- \frac{a|x|}{\lambda^{p}}
- C \frac{a|t_0|}{\lambda }
-C\frac{a^{2-\rho}|t_0|^\rho }{\lambda^{2-\rho}}
\ge \frac{1}{2}
$$
for $\lambda \ge \lambda_0$.
Hence we have 
$ |x^{(N+1)}(s)|\ge \frac{1}{2} |t^*|\lambda |\xi| \ge 
\frac{1}{2a} |t^*|\lambda$. 
\par
 In the same way as above, we can show that 
 $$
 |x^{(N+1)}(s)| \le 2 |t^*|\lambda a
 $$
 for $\lambda \ge \lambda_0$, $\lambda^{p-1}\le |t^*| \le t_0$, 
$x\in K$ and $\xi \in \Gamma $ with $1/a \le |\xi |\le a$, 
 assuming that \eqref{est-15} holds for $N$. 
\end{proof}
%
%


\begin{thebibliography}{99}
%
\bibitem{C-F-1}A.~C\'ordoba and C.~Fefferman, {\em Wave packets and
Fourier integral operators}, Comm. Partial Differential Equations {\bf
3} (1978), 979--1005.
%
\bibitem{C-K-S-1}W. Craig, T. Kappeler and W. Strauss, 
{\em Microlocal dispersive smoothing for the Schr\"odinger equations}, 
Commun. Pure and Appl. Math. {\bf 48} (1995), 760--860. 
%
\bibitem{Delort-1}
J.-M.~Delort,
F.B.I. transformation.
Second microlocalization and semilinear caustics. Lecture Notes in
	Mathematics, {\bf 1522}. Springer-Verlag, Berlin, 1992.
%
\bibitem{Doi-1}S. Doi, 
{\em Smoothing effects for Schr\"odinger evolution equation and global
behavior of geodesic flow}, Math. Ann. {\bf 318} (2000),
	355--389. 
%
\bibitem{Doi-2}S. Doi, 
{\em Commutator algebra and abstract smoothing effect}, 
J. Funct. Anal. {\bf 168} (1999), 428--469. 
%
%
\bibitem{Gerard-1}P.~G\'erard, 
{\em Moyennisation et r\'egrularit\'e deux-mikurolocale}, 
Ann. Sci. \'Ecole Norm. Sup. {\bf 23} (1990), 89-121. 
%
%
\bibitem{Folland-1}
G.~B.~Folland, 
Harmonic analysis in phase space, Prinston Univ. Press, 1989. 
%
\bibitem{Grochenig-1}
K. Gr\"ochenig, Foundations of Time-Frequency
Analysis, Birkh\"auser, Boston, 2001.
%
 \bibitem{H-W-1} A.~Hassell and J.~Wunsch,
 {\em The Schr\"odinger propagator for scattering metrics}, 
Ann. of math. {\bf 182}  (2005),  487--523.
%
\bibitem{Hoermander-1}
L. H\"ormander, The analysis of Linear Partial Differential Operators I,
Springer, Berlin, 1989.
\bibitem{Hoermander-2}
L.~H\"ormander, {\em Fourier integral operators I},
Acta. Math.{\bf 127} (1971), 79--183. 
%
\bibitem{K-K-I-1}
K. Kato, M. Kobayashi and S. Ito, {\em Representation of Schr\"odinger
	operator of a free particle via short time Fourier transform 
and its applications}, Tohoku Math. Journal {\bf 64}(2012), 223--231.
%
\bibitem{K-K-I-2}
K. Kato, M. Kobayashi and S. Ito, 
{\em Remark on wave front sets of solutions to Schr\"odinger equation of 
a free particle and a harmonic oscillator}, SUT J. Math. {\bf 47} (2011), 175--183.
%
%
\bibitem{K-K-I-3}
K. Kato, M. Kobayashi and S. Ito, 
{\em Remark on characterization of wave front set by wave packet
	transform}, arXiv:1408.1370v1.
%
%
\bibitem{K-K-I-4}
K. Kato, M. Kobayashi and S. Ito, 
{\em Application of wave packet transform to Schr\"odinger equations}, 
RIMS K\^oky\^uroku Bessatsu {\bf B33}, Harmonic analysis and nonlinear 
partial differential equations, 29--39. 
%
%
 \bibitem{Lascar-1} R.~Lascar,
 {\em Propagation des singularit\'e des solutions d'\'equations
pseudo-differentielles quasi homog\`enes}, 
Ann. Inst. Fourier, Grenoble {\bf 27}  (1977), 79--123.
%
\bibitem{Nakamura-1}
S. Nakamura, {\em Propagation of the homogeneous wave front set 
for Schr\"odinger equations}, Duke Math. J., {\bf 126} (2003), 349--367.
%
\bibitem{Nakamura-2}
S. Nakamura, {\em Semiclassical singularities propagation property for
Schr\"odinger equations},
J. Math. Soc. Japan, {\bf 61} (2009), 177--211.
%
\bibitem{Okaji-1}
T.~$\bar{\mathrm O}$kaji, {\em A note on the wave packet transforms}, 
Tsukuba J. Math.  {\bf 25}  (2001), 383--397. 
%

\bibitem{Okaji-2}
T.~$\bar{\mathrm O}$kaji, 
{\em Propagation of wave packets and its applications.}
	Operator Theory: Advances and Appl. 
	J. Math.  {\bf 126}  (2001), 239--243. 
%
 \bibitem{P-S-1} C.~Parenti and F.~Segala,
 {\em Propagation and reflection of singularities for a class 
of evolution equations}, 
Comm. Partial Differential Equations {\bf 6}  (1981), 741--782.
%
 \bibitem{Sakurai-1} T.~Sakurai,
 {\em Quasi-Homogeneous wave front set and fundamental solutions
for the Schr\"odinger Operator}, 
Sci. Papers of Coll. General Edu. {\bf 32}  (1982),
	 1--13.
%
\bibitem{Yajima-1} K.~Yajima, 
{\em Smoothness and nonsmoothness of the fundamental solution of time
	dependent Schr\"odinger equations}, 
Comm. Math. Phys. {\bf 181} (1996), 605--629.
\end{thebibliography}
\end{document}